\documentclass[11pt,reqno]{amsart}

\usepackage{amssymb,amsmath,amsfonts,amsthm}
\usepackage{microtype}
\usepackage[colorlinks=true,linkcolor=blue,citecolor=blue,urlcolor=blue]{hyperref}

\newtheorem{theorem}{Theorem}[section]
\newtheorem{lemma}[theorem]{Lemma}

\newtheorem{corollary}[theorem]{Corollary}
\theoremstyle{definition}
\newtheorem{definition}[theorem]{Definition}
\theoremstyle{remark}

\numberwithin{equation}{section}

\newcommand{\R}{\mathbb R}
\newcommand{\cD}{\mathcal D}
\newcommand{\cE}{\mathcal E}

\title[Volume growth and recurrence of fractional powers]
{Volume Growth and Recurrence of Fractional Powers of the
Laplace--Beltrami Operator}

\author{Haoxuan Cheng}
\address{School of Mathematical Sciences, Fudan University,
Shanghai 200433, China}
\email{hxcheng25@m.fudan.edu.cn}
\date{}


\begin{document}
\raggedbottom

\begin{abstract}
Let \(M\) be a connected geodesically complete Riemannian manifold without
boundary, write \(\mu\) for Riemannian volume, and set
\(V(o,r)=\mu(B(o,r))\) for geodesic balls centered at \(o\). For
\(0<\alpha<2\), let \(X^{(\alpha)}\) be the process obtained
by subordinating Brownian motion with an independent
\(\alpha/2\)-stable subordinator; its \(L^2\)-generator is
\(-(-\Delta)^{\alpha/2}\). We prove that
\[
  \int^\infty\frac{dt}{V(o,t^{1/\alpha})}=\infty
\]
implies that \(X^{(\alpha)}\) is recurrent. The proof uses a spectral trace
estimate and radial cutoffs on
\(M\times(0,\infty)\), where the auxiliary measure is
\(y^{1-\alpha}\,d\mu\,dy\). It proves the sufficient implication in
Grigor'yan's Problem~26.
\end{abstract}

\maketitle
\enlargethispage{3pt}

\section{Introduction}
\label{sec:introduction}

Let \((M,g)\) be a connected geodesically complete Riemannian manifold
without boundary.  We denote its Riemannian volume measure by \(\mu\), fix a
point \(o\in M\), and write
\[
 B(o,r)=\{x\in M:d_g(o,x)<r\},
 \qquad V(o,r)=\mu(B(o,r)).
\]
Here \(d_g\) is the geodesic distance.  Let \(A=-\Delta\) be the
nonnegative Laplace--Beltrami operator on \(L^2(M,\mu)\).  Given
\(0<\alpha<2\), the fractional operator in this paper is
\(A^{\alpha/2}=(-\Delta)^{\alpha/2}\), defined by spectral calculus.  Its
negative generates the symmetric Markov process obtained by evaluating
Brownian motion at an independent \(\alpha/2\)-stable random time.  This is
the \(\alpha\)-process in Grigor'yan's terminology.  The operator, the random
time change, and the associated Dirichlet form are defined in
Section~\ref{sec:preliminaries}.

The problem is to decide recurrence from the single radial volume function
\(V(o,r)\).  In \cite[Problem~26]{Grigoryan1999}, Grigor'yan asked whether
the implication
\begin{equation}
 \int^\infty \frac{dt}{V(o,t^{1/\alpha})}=\infty
 \quad\Longrightarrow\quad
 \text{the \(\alpha\)-process is recurrent}
 \label{eq:original-condition}
\end{equation}
holds on every geodesically complete manifold, without a relative
Faber--Krahn inequality.  A relative Faber--Krahn inequality gives a uniform
lower bound for the first Dirichlet spectral value of every subdomain of a
ball in terms of the subdomain's relative volume.  It is therefore a uniform
local spectral hypothesis that supplies information beyond the volumes of the
concentric balls \(B(o,r)\).

Two results in \cite{Grigoryan1999} provide the relevant background.
Theorem~16.2 proves recurrence under the pointwise bound
\(V(o,r)\leq Cr^\alpha\) for all sufficiently large \(r\).  Theorem~16.3
proves that the integral in \eqref{eq:original-condition} is necessary and
sufficient when a relative Faber--Krahn inequality is available.
Grigor'yan returned to the question in his 2005 lecture notes, published in
\cite{Grigoryan2006}.  Immediately after the corresponding
Corollary~9.17, he again stated that the sufficient implication without a
relative Faber--Krahn inequality was not known, except when \(\alpha=2\).

McGillivray proved volume-growth criteria for subordinated strongly local
Dirichlet forms under a strong radial-symmetry condition
\cite{McGillivray1995,McGillivray1997}.  The explicit derivative criterion in
\cite{McGillivray1995} assumes
\(1<\alpha<2\) and
\(\dot v(r)\leq Cr^{\alpha-1-\varepsilon}\) for some
\(0<\varepsilon\leq\alpha-1\), where \(v\) is the radial volume function in
that criterion.  Problem~26 concerns arbitrary complete manifolds and also
allows critical logarithmic volume growth through
\eqref{eq:original-condition}.

Other recurrence criteria for subordinated and jump processes use auxiliary
rate functions, capacity estimates, or quantitative bounds on the jump
kernel.  {\^O}kura developed rate-function criteria and structural results for
subordinated symmetric Markov processes
\cite{Okura2002}; related capacitary estimates appear in
\cite{Okura2003}.  Kaneko treated recurrence and transience of symmetric Hunt
processes \cite{Kaneko2000}.  Masamune--Uemura--Wang assume a uniform
truncated second-moment bound
\cite[Theorem~1.2]{MasamuneUemuraWang2012}, and {\^O}kura--Uemura impose
volume and jump-kernel hypotheses on classes of pure-jump forms
\cite{OkuraUemura2015}.  These criteria use data beyond the one-point volume
function in \eqref{eq:original-condition}.

Fractional powers admit weighted extension formulations; see Stinga--Torrea
\cite{StingaTorrea2010} and the
Dirichlet-space treatment of Baudoin--Lang--Sire
\cite{BaudoinLangSire2020}.  For local diffusion forms, Sturm derived
volume-capacity estimates and recurrence criteria from radial cutoffs
\cite{Sturm1995}.  We prove the required trace estimate directly and construct
radial cutoffs on the weighted product space.  Together, these steps convert
the integral in \eqref{eq:original-condition} into vanishing fractional
energy.

For the statement below, put \(s=\alpha/2\) and let
\[
 \cE_s(f,f)=\|A^{s/2}f\|_2^2,
 \qquad \cD(\cE_s)=\cD(A^{s/2}).
\]
This is the quadratic form associated with the \(\alpha\)-process.  We call
the process recurrent when there are \(u_j\in\cD(\cE_s)\) such that
\(0\leq u_j\leq1\), \(u_j\to1\) almost everywhere, and
\(\cE_s(u_j,u_j)\to0\).  Section~\ref{subsec:subordination} verifies the
association with the process and records the standard Dirichlet-form
criterion.  The following theorem answers Problem~26.

\begin{theorem}
\label{thm:main}
Let \((M,g)\) be a connected geodesically complete Riemannian manifold
without boundary, let \(o\in M\), let \(0<\alpha<2\), and put
\(s=\alpha/2\).  If
\[
 \int^\infty \frac{dt}{V(o,t^{1/\alpha})}=\infty,
\]
then the \(\alpha\)-process is recurrent.  Equivalently, its associated
quadratic form \(\cE_s\) admits the zero-energy approximating sequence just
described.
\end{theorem}

The proof is variational.  Put \(s=\alpha/2\).  For an auxiliary function
\(U:M\times(0,\infty)\to\mathbb R\), we first prove that the fractional
energy of its boundary trace \(U(\cdot,0)\) is bounded by
\[
 \int_0^\infty y^{1-2s}
 \bigl(\|\partial_yU(\cdot,y)\|_2^2
       +\|A^{1/2}U(\cdot,y)\|_2^2\bigr)\,dy.
\]
This weighted quantity is the local Dirichlet energy of \(U\) on the product
space with measure \(y^{1-2s}\,d\mu\,dy\).  We construct radial cutoffs on
that weighted product, using successive annuli whose radii double.  On each
annulus, the drop of the cutoff is chosen in proportion to the reciprocal of
the annulus's quadratic energy cost.  The volume integral in
\eqref{eq:original-condition} makes the sum of these reciprocal costs
diverge.  The boundary traces then converge to \(1\), and their fractional
energies tend to zero.  The Dirichlet-form recurrence criterion completes
the proof.

Section~\ref{sec:preliminaries} fixes the analytic and probabilistic
framework.  Section~\ref{sec:trace} proves the trace inequality.
Section~\ref{sec:cutoffs} constructs and estimates the weighted cutoffs, and
Section~\ref{sec:recurrence} proves Theorem~\ref{thm:main} and records its
model consequences.
\section{Analytic and probabilistic framework}
\label{sec:preliminaries}

\subsection{The Laplacian and its fractional Dirichlet form}
\label{subsec:laplacian}

Throughout the paper, \((M,g)\), \(d_g\), \(\mu\), \(o\), \(B(o,r)\),
and \(V(o,r)\) have the meanings fixed in the introduction.  By the
Hopf--Rinow theorem, every closed bounded ball in \(M\) is compact.  In
particular, \(V(o,r)<\infty\) for finite \(r\).

We work with the real Hilbert space
\[
 H=L^2(M,\mu;\mathbb R),
 \qquad
 \langle f,h\rangle_2=\int_Mfh\,d\mu,
 \qquad
 \|f\|_2=\langle f,f\rangle_2^{1/2}.
\]
Its complexification is used when applying the spectral theorem.  Let
\(\nabla\) be the Riemannian gradient and use the convention
\(\Delta=\operatorname{div}\nabla\).  The operator
\[
 A=-\Delta\geq0
\]
is the nonnegative self-adjoint realization associated with the closure of
\begin{equation}
 \cE_0(f,h)=\int_M\langle\nabla f,\nabla h\rangle_g\,d\mu,
 \qquad f,h\in C_c^\infty(M).
 \label{eq:laplace-form}
\end{equation}
Here \(C_c^\infty(M)\) denotes the smooth compactly supported functions.
On a complete manifold this realization agrees with the closure of the
Laplace--Beltrami operator initially defined on \(C_c^\infty(M)\); see
\cite[Section~2.2]{Grigoryan1999}.

Let \(W^{1,2}(M)\) be the space of \(L^2\)-functions whose weak gradient is
in \(L^2\).  Completeness gives
\begin{equation}
 \cD(A^{1/2})=W^{1,2}(M),
 \qquad
 \|A^{1/2}u\|_2^2=\int_M|\nabla u|_g^2\,d\mu.
 \label{eq:laplace-form-domain}
\end{equation}
The following density argument is used below.  Choose a Lipschitz function
\(\theta:[0,\infty)\to[0,1]\) equal to \(1\) on \([0,1]\), equal to
\(0\) on \([2,\infty)\), and satisfying \(|\theta'|\leq2\) almost
everywhere.  Then
\(\chi_R(x)=\theta(d_g(o,x)/R)\) has compact support and
\(|\nabla\chi_R|\leq2/R\) almost everywhere.  If
\(u\in W^{1,2}(M)\), dominated convergence and
\(\|u\nabla\chi_R\|_2\leq2R^{-1}\|u\|_2\) show that
\(\chi_Ru\to u\) in \(W^{1,2}(M)\).  A partition of unity and local
mollification then approximate each compactly supported Sobolev function by
members of \(C_c^\infty(M)\).  Thus \(C_c^\infty(M)\) is a form core for
\eqref{eq:laplace-form}: it is dense in \(W^{1,2}(M)\) for the norm
\((\|u\|_2^2+\|A^{1/2}u\|_2^2)^{1/2}\).

Let \(E\) be the projection-valued spectral measure of \(A\).  For a Borel
set \(I\subset[0,\infty)\), \(E(I)=\mathbf1_I(A)\) is the corresponding
orthogonal projection.  If \(\gamma>0\), then
\[
 A^\gamma f=\int_{[0,\infty)}\lambda^\gamma\,dE(\lambda)f,
\]
with domain
\begin{equation}
 \begin{aligned}
 \cD(A^\gamma)
 &=\left\{f\in L^2(M,\mu;\mathbb C):
   \int_{[0,\infty)}\lambda^{2\gamma}\,d\mu_f(\lambda)<\infty
   \right\},\\
 \mu_f(I)&=\|E(I)f\|_2^2.
 \end{aligned}
 \label{eq:spectral-power-domain}
\end{equation}
The measure \(\mu_f\) is the scalar spectral measure of \(f\).  Spectral
projections and powers preserve the real subspace.

Fix \(0<\alpha<2\) and put
\begin{equation}
 s=\frac{\alpha}{2}\in(0,1).
 \label{eq:s-alpha}
\end{equation}
The closed quadratic form of \(A^s=(-\Delta)^{\alpha/2}\) is
\begin{equation}
 \cE_s(f,h)=\langle A^{s/2}f,A^{s/2}h\rangle_2,
 \qquad
 \cD(\cE_s)=\cD(A^{s/2})\cap H.
 \label{eq:fractional-form}
\end{equation}
In particular,
\begin{equation}
 \cE_s(f,f)=\int_{[0,\infty)}\lambda^s\,d\mu_f(\lambda).
 \label{eq:fractional-spectral-energy}
\end{equation}
The form domain \(\cD(A^{s/2})\) is larger than the operator
domain \(\cD(A^s)\).

\subsection{The subordinated process and the recurrence criterion}
\label{subsec:subordination}

The proof from Section~\ref{sec:trace} onward is formulated in terms of
quadratic forms.  This subsection identifies \(\cE_s\) as the form of the
\(\alpha\)-process and states the zero-energy approximation criterion used in
Section~\ref{sec:recurrence}.

Let \((B_t)_{t\geq0}\) be Brownian motion on \(M\), with generator
\(\Delta\) and lifetime \(\zeta\).  Thus our normalization runs at twice
the speed of the convention with generator \(\frac12\Delta\).  Under the
latter convention, stable subordination gives the generator
\(-2^{-s}A^s\); the deterministic time change \(t\mapsto2^st\) gives
\(-A^s\).  Such a positive constant rescaling of time does not affect
recurrence.  For a
bounded Borel function \(f\), its transition operator is
\[
 \mathsf P_t f(x)=\mathbb E_x[f(B_t);t<\zeta],
\]
where \(\mathbb E_x\) denotes expectation for Brownian motion started from
\(x\).  The induced strongly continuous contraction semigroup on \(H\) is
\(e^{-tA}\).

An \(s\)-stable subordinator \((S_t)_{t\geq0}\) is a nondecreasing
right-continuous process with stationary independent increments whose
Laplace transform is
\begin{equation}
 \mathbb E(e^{-\lambda S_t})=e^{-t\lambda^s},
 \qquad \lambda,t>0.
 \label{eq:stable-clock}
\end{equation}
Assume that \(S\) and \(B\) are independent, and let \(\eta_t\) be the law
of \(S_t\).  The \(\alpha\)-process is
\(X_t^{(\alpha)}=B_{S_t}\), with the convention that \(B\) stays at the cemetery
state after its lifetime.  Its pointwise transition operator is
\begin{equation}
 \mathsf P_t^{(s)}f(x)
 =\mathbb E_x[f(B_{S_t});S_t<\zeta]
 =\int_{(0,\infty)}\mathsf P_r f(x)\,\eta_t(dr).
 \label{eq:pointwise-subordination}
\end{equation}
On \(H\), the spectral theorem and \eqref{eq:stable-clock} give
\begin{equation}
 T_t^{(s)}=\int_{(0,\infty)}e^{-rA}\,\eta_t(dr)=e^{-tA^s}.
 \label{eq:l2-subordination}
\end{equation}
Hence the \(L^2\)-generator is \(-A^s=-(-\Delta)^{\alpha/2}\), and its
closed form is precisely \eqref{eq:fractional-form}.  Each operator
\(e^{-rA}\) is Markovian, and \(\eta_t\) is a probability measure.  Thus
\(T_t^{(s)}\) is a symmetric Markov semigroup, so
\((\cE_s,\cD(\cE_s))\) is a symmetric Dirichlet form.  Formula
\eqref{eq:l2-subordination} is called stable, or Bochner, subordination: it
averages the original semigroup \(e^{-rA}\) over the probability law
\(\eta_t\) of the random clock.  See \cite{Okura2002,FOT2011}.

A symmetric Dirichlet form on \(H\) is a densely defined,
closed, nonnegative symmetric bilinear form \((\cE,\cD(\cE))\) satisfying
the Markov property
\[
 u\in\cD(\cE)
 \quad\Longrightarrow\quad
 (0\vee u)\wedge1\in\cD(\cE),
 \qquad
 \cE((0\vee u)\wedge1,(0\vee u)\wedge1)\leq\cE(u,u).
\]
The form is \emph{regular} if
\(\cD(\cE)\cap C_c(M)\) is dense in \(\cD(\cE)\) for the norm
\((\|u\|_2^2+\cE(u,u))^{1/2}\), and dense in \(C_0(M)\) for the uniform
norm.  Here \(C_c(M)\) denotes the continuous compactly supported
functions, and \(C_0(M)\) the continuous functions vanishing at infinity.

The fractional form \((\cE_s,\cD(\cE_s))\) is regular.  If
\(f\in\cD(A^{s/2})\), the spectral truncations
\(E([0,R])f\) belong to \(\cD(A^{1/2})\) and converge to \(f\) in the
graph norm
\((\|f\|_2^2+\|A^{s/2}f\|_2^2)^{1/2}\) of \(A^{s/2}\).  The density statement following
\eqref{eq:laplace-form-domain}, together with
\(\lambda^s\leq1+\lambda\), then shows that \(C_c^\infty(M)\) is a form
core for \(\cE_s\).  Its uniform density in \(C_0(M)\) proves regularity.

Because \(M\) is connected, its minimal heat kernel is strictly positive
for every positive time; see \cite[Section~2.2]{Grigoryan1999}.  Since
\[
 \mathbb P(S_t=0)=\lim_{\lambda\to\infty}
 \mathbb E(e^{-\lambda S_t})
 =\lim_{\lambda\to\infty}e^{-t\lambda^s}=0,
\]
formula
\eqref{eq:pointwise-subordination} shows that \(T_t^{(s)}\) maps every
nonzero nonnegative function to a positive function.  Thus the form has no
nontrivial invariant measurable subset.  The Hunt process associated with
this regular form is the subordinated process \(X^{(\alpha)}\) constructed
above.  Consequently, recurrence of the form is recurrence of the process in
Problem~26.

We use the following standard characterization of recurrence for regular
symmetric Dirichlet forms.

\begin{definition}
\label{def:recurrence}
The form \((\cE_s,\cD(\cE_s))\) is recurrent if there are functions
\(u_j\in\cD(\cE_s)\) such that
\begin{equation}
 0\leq u_j\leq1,
 \qquad u_j\longrightarrow1\quad\mu\text{-a.e.},
 \qquad \cE_s(u_j,u_j)\longrightarrow0.
 \label{eq:recurrence-criterion}
\end{equation}
\end{definition}

Condition \eqref{eq:recurrence-criterion} is equivalent to the usual
Dirichlet-form definition of recurrence by
\cite[Theorem~1.6.3]{FOT2011}.  The phrase ``the \(\alpha\)-process is
recurrent'' will always mean recurrence of its associated form in this
sense.

To put the volume hypothesis into the scale used by the cutoff construction,
make the substitution \(t=r^\alpha\).  It gives
\begin{equation}
 \int_T^\infty\frac{dt}{V(o,t^{1/\alpha})}
 =\alpha\int_{T^{1/\alpha}}^\infty
 \frac{r^{\alpha-1}}{V(o,r)}\,dr.
 \label{eq:change-of-variables}
\end{equation}
Since \(\alpha=2s\), the hypothesis of Theorem~\ref{thm:main} is
equivalent to
\begin{equation}
 \int^\infty\frac{r^{2s-1}}{V(o,r)}\,dr=\infty.
 \label{eq:radial-condition}
\end{equation}
\section{A weighted trace estimate}
\label{sec:trace}

We first establish the trace estimate that transfers weighted energy on an
auxiliary half-line to the fractional energy of the boundary value.

We begin by fixing the meaning of integration and differentiation for a
Hilbert-space-valued map.  Let \(H\) be a real or complex Hilbert space.  A
map \(G:(a,b)\to H\) is \emph{strongly measurable} if it is almost
everywhere the pointwise limit of measurable simple \(H\)-valued maps.  It is
\emph{Bochner integrable} if it is strongly measurable and
\(\int_a^b\|G(y)\|_H\,dy<\infty\).  Its Bochner integral is the vector in
\(H\) characterized by
\[
 \left\langle\int_a^bG(y)\,dy,h\right\rangle_H
 =\int_a^b\langle G(y),h\rangle_H\,dy,
 \qquad h\in H.
\]
A map \(F:(0,\infty)\to H\) is locally absolutely continuous if, on every
compact interval \([a,b]\subset(0,\infty)\), there is a Bochner-integrable
map \(G\) such that
\[
 F(y)=F(a)+\int_a^yG(z)\,dz,
 \qquad a\leq y\leq b.
\]
In that case the norm derivative
\(F'(y)=\lim_{h\to0}(F(y+h)-F(y))/h\) exists for almost every \(y\), equals
\(G(y)\), and satisfies the Hilbert-valued fundamental theorem of calculus.
Derivatives understood in this sense will be called Bochner derivatives.

\begin{lemma}
\label{lem:hilbert-trace}
Let \(H\) be a real or complex Hilbert space. Suppose that
\(F:(0,\infty)\to H\) is locally absolutely continuous and has a strong
limit \(F(0)=\lim_{y\downarrow0}F(y)\). If \(\lambda>0\) and
\[
 \int_0^\infty y^{1-2s}
 \bigl(\|F'(y)\|_H^2+\lambda\|F(y)\|_H^2\bigr)\,dy<\infty,
\]
then
\begin{equation}
 \lambda^s\|F(0)\|_H^2
 \leq C_s\int_0^\infty y^{1-2s}
 \bigl(\|F'(y)\|_H^2+\lambda\|F(y)\|_H^2\bigr)\,dy.
 \label{eq:hilbert-scaled}
\end{equation}
\end{lemma}

\begin{proof}
At unit scale, choose \(r\in[1,2]\) such that
\[
 \|F(r)\|_H^2\leq C_s\int_1^2z^{1-2s}\|F(z)\|_H^2\,dz.
\]
This follows from averaging over \([1,2]\), because the weight
\(z^{1-2s}\) is comparable to \(1\) there.
For \(0<\varepsilon<r\), the Bochner fundamental theorem of calculus gives
\[
 \|F(\varepsilon)\|_H
 \leq\|F(r)\|_H+\int_\varepsilon^r\|F'(z)\|_H\,dz.
\]
The derivative term has the explicit bound
\begin{align*}
 \left(\int_\varepsilon^r\|F'(z)\|_H\,dz\right)^2
 &\leq
 \left(\int_\varepsilon^r z^{1-2s}\|F'(z)\|_H^2\,dz\right)
 \left(\int_\varepsilon^r z^{2s-1}\,dz\right)\\
 &\leq \frac{2^{2s}}{2s}
 \int_0^2z^{1-2s}\|F'(z)\|_H^2\,dz.
\end{align*}
After squaring the preceding triangle inequality and absorbing the two
terms into a constant depending only on \(s\), we obtain
\[
 \|F(\varepsilon)\|_H^2
 \leq C_s\int_0^\infty z^{1-2s}
 \bigl(\|F'(z)\|_H^2+\|F(z)\|_H^2\bigr)\,dz.
\]
The assumed strong limit at zero means convergence in the norm of \(H\), so
letting \(\varepsilon\downarrow0\) proves the unit-scale estimate.

For the general scale, set \(G(z)=F(z/\sqrt\lambda)\).  Then
\(G'(z)=\lambda^{-1/2}F'(z/\sqrt\lambda)\) almost everywhere, and the change
of variables \(z=\sqrt\lambda y\) gives
\begin{align*}
 \int_0^\infty z^{1-2s}\|G'(z)\|_H^2\,dz
 &=\lambda^{-s}\int_0^\infty
   y^{1-2s}\|F'(y)\|_H^2\,dy,\\
 \int_0^\infty z^{1-2s}\|G(z)\|_H^2\,dz
 &=\lambda^{1-s}\int_0^\infty
   y^{1-2s}\|F(y)\|_H^2\,dy.
\end{align*}
Multiplying the unit-scale estimate for \(G\) by \(\lambda^s\) proves
\eqref{eq:hilbert-scaled}.
\end{proof}

Let \(A\) be a nonnegative self-adjoint operator on \(H\).  For a
path \(U:(0,\infty)\to H\) with the regularity specified in the next lemma,
we introduce the notation
\begin{equation}
 \mathfrak E_s(U)=\int_0^\infty y^{1-2s}
 \bigl(\|U'(y)\|_H^2+\|A^{1/2}U(y)\|_H^2\bigr)\,dy.
 \label{eq:extension-energy}
\end{equation}
We call \(\mathfrak E_s(U)\) the \emph{extension energy}.  The weight
\(y^{1-2s}\) is the standard
fractional-extension weight; compare
\cite{StingaTorrea2010,BaudoinLangSire2020}.  When \(H=L^2(M,\mu)\) and
\(A=-\Delta\), formula \eqref{eq:laplace-form-domain} turns
\eqref{eq:extension-energy} into the local weighted Dirichlet energy on
\(M\times(0,\infty)\); this identification is written explicitly in
\eqref{eq:geometric-extension-energy} below.

\begin{lemma}
\label{lem:spectral-trace}
Let \(A\geq0\) be self-adjoint on a real or complex Hilbert space \(H\).
Suppose that \(U:(0,\infty)\to H\) is locally absolutely continuous, has a
strong limit \(U(0)\) at zero, and satisfies the following conditions:
\begin{enumerate}
 \item \(U(y)\in\cD(A^{1/2})\) for almost every \(y>0\);
 \item the map \(y\mapsto A^{1/2}U(y)\), defined arbitrarily on the
 exceptional null set, is strongly measurable;
 \item the extension energy \(\mathfrak E_s(U)\) defined in
 \eqref{eq:extension-energy} is finite.
\end{enumerate}
Then \(U(0)\in\cD(A^{s/2})\) and
\begin{equation}
 \|A^{s/2}U(0)\|_H^2\leq C_s\mathfrak E_s(U).
 \label{eq:spectral-trace}
\end{equation}
\end{lemma}

\begin{proof}
Let \(E\) be the spectral measure of \(A\), and write \(\mathbb Z\) for the
integers. For \(j\in\mathbb Z\), set
\[
 P_j=E((2^j,2^{j+1}]),
 \qquad U_j(y)=P_jU(y).
\]
Since \(P_j\) is bounded, \(U_j\) is locally absolutely continuous,
\(U_j'=P_jU'\) almost everywhere, and
\(U_j(y)\to P_jU(0)\) strongly as \(y\downarrow0\). On
\(\operatorname{Ran}P_j\), the range of the projection \(P_j\), the
spectral theorem gives the quadratic-form inequalities
\begin{equation}
 A^sP_j\leq2^{(j+1)s}P_j,
 \qquad
 2^jP_j\leq AP_j.
 \label{eq:band-inequalities}
\end{equation}
Thus, for example, the first inequality means
\(\|A^{s/2}h\|_H^2\leq2^{(j+1)s}\|h\|_H^2\) for
\(h\in\operatorname{Ran}P_j\); neither inequality is asserted on all of
\(H\).
Moreover,
\[
 \int_0^\infty y^{1-2s}
 \bigl(\|P_jU'(y)\|_H^2+2^j\|P_jU(y)\|_H^2\bigr)\,dy
 \leq\mathfrak E_s(U)<\infty,
\]
where the second term is controlled by \(2^jP_j\leq AP_j\).  Thus all
hypotheses of Lemma~\ref{lem:hilbert-trace} hold for \(U_j\).
Apply Lemma~\ref{lem:hilbert-trace} to \(U_j\) with
\(\lambda=2^j\). For \(N\geq1\), define the finite spectral truncation
\begin{equation}
 Q_N=E((2^{-N},2^{N+1}])=\sum_{j=-N}^{N}P_j.
 \label{eq:finite-spectral-truncation}
\end{equation}
The vector \(Q_NU(0)\) has bounded spectral support and therefore belongs to
\(\cD(A^{s/2})\). Orthogonality, the Hilbert-valued trace estimate, and
\eqref{eq:band-inequalities} yield
\begin{align*}
 \|A^{s/2}Q_NU(0)\|_H^2
 &=\sum_{j=-N}^{N}\|A^{s/2}P_jU(0)\|_H^2\\
 &\leq2^s\sum_{j=-N}^{N}2^{js}\|P_jU(0)\|_H^2\\
 &\leq C_s\int_0^\infty y^{1-2s}
       \sum_{j=-N}^{N}\|P_jU'(y)\|_H^2\,dy\\
 &\quad+C_s\int_0^\infty y^{1-2s}
       \sum_{j=-N}^{N}2^j\|P_jU(y)\|_H^2\,dy\\
 &\leq C_s\mathfrak E_s(U).
\end{align*}
The last line follows pointwise for almost every \(y\)
from orthogonality and the lower band estimate:
\begin{align*}
 \sum_{j=-N}^{N}\|P_jU'(y)\|_H^2
 &=\|Q_NU'(y)\|_H^2\leq\|U'(y)\|_H^2,\\
 \sum_{j=-N}^{N}2^j\|P_jU(y)\|_H^2
 &\leq\sum_{j=-N}^{N}\|A^{1/2}P_jU(y)\|_H^2\\
 &=\|A^{1/2}Q_NU(y)\|_H^2
 \leq\|A^{1/2}U(y)\|_H^2.
\end{align*}
The spectral projections commute with \(A^{1/2}\), and Tonelli's theorem
permits the finite sums to be integrated against \(y^{1-2s}\,dy\).
As \(N\to\infty\), the intervals
\((2^{-N},2^{N+1}]\) increase to \((0,\infty)\). Monotone convergence for
the scalar spectral measure of \(U(0)\) gives
\[
 \int_{(0,\infty)}\lambda^s\,d\mu_{U(0)}(\lambda)
 \leq C_s\mathfrak E_s(U).
\]
The zero spectral subspace contributes nothing.  The finiteness of this
integral proves \(U(0)\in\cD(A^{s/2})\), and its value is
\(\|A^{s/2}U(0)\|_H^2\).  This proves
\eqref{eq:spectral-trace}.
\end{proof}
\section{Weighted volume and radial cutoffs}
\label{sec:cutoffs}

We apply Lemma~\ref{lem:spectral-trace} with
\(H=L^2(M,\mu;\R)\) and \(A=-\Delta\). The auxiliary coordinate is denoted
by \(y\geq0\). Equip \(M\times(0,\infty)\) with the product metric
\(g+dy^2\), and define
\begin{equation}
 \rho(x,y)=\sqrt{d_g(o,x)^2+y^2},
 \qquad
 d\widetilde\mu(x,y)=y^{1-2s}\,d\mu(x)\,dy.
 \label{eq:product-objects}
\end{equation}
Thus \(\widetilde\mu\) is a weighted product measure.  Since
\(s\in(0,1)\), its density is locally integrable at \(y=0\).

For \(R>0\), let
\begin{equation}
 W(R)=\widetilde\mu\{(x,y)\in M\times(0,\infty):\rho(x,y)<R\}.
 \label{eq:weighted-volume-definition}
\end{equation}
If \(U\) has the Sobolev regularity used below, then
\eqref{eq:laplace-form-domain} identifies the extension energy with
\begin{equation}
 \mathfrak E_s(U)=
 \int_0^\infty\!\int_M y^{1-2s}
 \bigl(|\partial_yU(x,y)|^2+|\nabla_xU(x,y)|_g^2\bigr)
 \,d\mu(x)\,dy.
 \label{eq:geometric-extension-energy}
\end{equation}
Here \(\nabla_x\) differentiates only in the \(M\)-variable.

\begin{lemma}
\label{lem:weighted-volume}
For every \(R>0\),
\begin{equation}
 W(R)\leq\frac{R^{2-2s}}{2-2s}V(o,R).
 \label{eq:weighted-volume}
\end{equation}
\end{lemma}

\begin{proof}
Fubini's theorem and the monotonicity of \(V(o,\cdot)\) give
\begin{align*}
 W(R)
 &=\int_0^Ry^{1-2s}
   V\!\left(o,\sqrt{R^2-y^2}\right)\,dy\\
 &\leq V(o,R)\int_0^Ry^{1-2s}\,dy
 =\frac{R^{2-2s}}{2-2s}V(o,R).
\end{align*}
\end{proof}

Fix \(R_0>0\), and let \(R_k=2^kR_0\) for \(k=0,1,\ldots\). Define
\begin{equation}
 a_k=\frac{R_k^{2s}}{V(o,R_k)},
 \qquad
 q_k=\frac{R_k^2}{W(R_{k+1})}.
 \label{eq:ak-qk}
\end{equation}
The sequence \((a_k)\) discretizes \eqref{eq:radial-condition}.  The
energy estimate below contains
\(W(R_{k+1})/R_k^2=q_k^{-1}\) on the annulus
\(R_k<\rho<R_{k+1}\).  Each \(q_k\) is finite and positive.

\begin{lemma}
\label{lem:series-divergence}
If \eqref{eq:radial-condition} holds, then
\[
 \sum_{k=0}^\infty a_k=\infty,
 \qquad
 \sum_{k=0}^\infty q_k=\infty.
\]
\end{lemma}

\begin{proof}
Since \(V(o,r)\) is nondecreasing,
\begin{align*}
 \int_{R_k}^{R_{k+1}}
 \frac{r^{2s-1}}{V(o,r)}\,dr
 &\leq\frac1{V(o,R_k)}
 \int_{R_k}^{2R_k}r^{2s-1}\,dr\\
 &=\frac{2^{2s}-1}{2s}\,a_k.
\end{align*}
Thus convergence of \(\sum_k a_k\) would imply convergence of the tail
integral, proving the first assertion by contraposition. Moreover,
Lemma~\ref{lem:weighted-volume} gives the exact comparison
\begin{equation}
 q_k\geq\frac{1-s}{2}\,
 \frac{R_{k+1}^{2s}}{V(o,R_{k+1})}
 =\frac{1-s}{2}\,a_{k+1}.
 \label{eq:qk-ak-comparison}
\end{equation}
Hence \(\sum_kq_k\) also diverges.
\end{proof}

For integers \(0\leq m<n\), set
\begin{equation}
 S_{m,n}=\sum_{k=m}^nq_k,
 \qquad
 \delta_k=\frac{q_k}{S_{m,n}}.
 \label{eq:drop-definition}
\end{equation}
The numbers \(\delta_k\) are positive and sum to \(1\). Let
\(\eta_{m,n}:[0,\infty)\to[0,1]\) be the continuous piecewise affine
function that equals \(1\) on \([0,R_m]\), decreases by \(\delta_k\) on
\([R_k,R_{k+1}]\), and equals \(0\) on
\([R_{n+1},\infty)\). Since \(R_{k+1}-R_k=R_k\), its slope on the
\(k\)-th interval is \(-\delta_k/R_k\). Define
\begin{equation}
 U_{m,n}(x,y)=\eta_{m,n}(\rho(x,y)).
 \label{eq:cutoff}
\end{equation}

Among all real numbers \(\varepsilon_k\) satisfying
\(\sum_{k=m}^n\varepsilon_k=1\), weighted Cauchy--Schwarz gives
\[
 \sum_{k=m}^n\frac{\varepsilon_k^2}{q_k}
 \geq\frac1{S_{m,n}},
\]
with equality for \(\varepsilon_k=\delta_k\). Thus the chosen drops minimize
the quadratic upper bound used below.

\begin{lemma}
\label{lem:cutoff-energy}
The function \(U_{m,n}\) satisfies all hypotheses of
Lemma~\ref{lem:spectral-trace} and
\begin{equation}
 \mathfrak E_s(U_{m,n})\leq\frac1{S_{m,n}}.
 \label{eq:cutoff-energy}
\end{equation}
\end{lemma}

\begin{proof}
Write \(R=R_{n+1}\). Since \(\eta_{m,n}\) vanishes on \([R,\infty)\),
\begin{equation}
 \operatorname{supp}U_{m,n}
 \subset\{(x,y):\rho(x,y)\leq R\}
 \subset\overline{B(o,R)}\times[0,R].
 \label{eq:cutoff-support}
\end{equation}
The last set is compact by Hopf--Rinow.

The distance function \(d_g(o,\cdot)\) is \(1\)-Lipschitz. Consequently
\(\rho\) is locally Lipschitz on the product space and
\(|\nabla_{g+dy^2}\rho|\leq1\) almost everywhere.  More explicitly, put
\(r(x)=d_g(o,x)\).  At every point where \(r\) is differentiable and
\(\rho>0\),
\[
 \nabla_x\rho=\frac{r}{\rho}\nabla r,
 \qquad
 \partial_y\rho=\frac{y}{\rho},
 \qquad
 |\nabla_x\rho|_g^2+|\partial_y\rho|^2=1,
\]
because \(|\nabla r|_g=1\) almost everywhere away from \(o\).  The global
Lipschitz bound and Rademacher's theorem give
\(|\nabla_{g+dy^2}\rho|\leq1\) almost everywhere.
Let \(L\) be the Lipschitz constant of \(\eta_{m,n}\).  For each
\(y\geq0\), the slice
\(U_{m,n}(\cdot,y)\) is a compactly supported Lipschitz function on \(M\),
so it lies in \(W^{1,2}(M)=\cD(A^{1/2})\). Moreover,
\[
 |\rho(x,y_1)-\rho(x,y_2)|\leq|y_1-y_2|
\]
and both slices vanish outside \(B(o,R)\). Hence
\begin{equation}
 \|U_{m,n}(\cdot,y_1)-U_{m,n}(\cdot,y_2)\|_2
 \leq L V(o,R)^{1/2}|y_1-y_2|.
 \label{eq:l2-slice-lipschitz}
\end{equation}
Thus \(y\mapsto U_{m,n}(\cdot,y)\) is an \(L^2(M)\)-valued Lipschitz map
and has the strong trace
\[
 U_{m,n}(\cdot,0)=\eta_{m,n}(d_g(o,\cdot)).
\]

The Lipschitz chain rule applies to
\(U_{m,n}=\eta_{m,n}\circ\rho\) almost everywhere.  For each fixed \(x\),
the equation \(\rho(x,y)=R_k\) has at most one solution \(y>0\).  A singleton
has zero measure for \(y^{1-2s}\,dy\), so Fubini's theorem shows that the
finitely many break sets have zero \(\widetilde\mu\)-measure.  On the annulus
\(R_k<\rho<R_{k+1}\),
\begin{equation}
 |\nabla_xU_{m,n}|_g^2+|\partial_yU_{m,n}|^2
 \leq\frac{\delta_k^2}{R_k^2}
 \quad\text{a.e.}
 \label{eq:annular-gradient}
\end{equation}

The pointwise derivative \(\partial_yU_{m,n}(x,y)\) exists for almost every
\((x,y)\).  By Fubini's theorem, for almost every fixed \(y>0\) it exists for
almost every \(x\).  For such a \(y\), the difference quotients satisfy
\[
 \frac{U_{m,n}(x,y+h)-U_{m,n}(x,y)}{h}
 \longrightarrow \partial_yU_{m,n}(x,y)
 \quad\text{for almost every }x
\]
as \(h\to0\), and their absolute values are bounded by
\(L\mathbf1_{B(o,R)}(x)\).  This dominating function belongs to
\(L^2(M,\mu)\).  Dominated convergence therefore gives convergence of the
difference quotients in \(L^2(M,\mu)\).  Hence the strong derivative of the
Hilbert-valued slice map exists almost everywhere and equals
\(\partial_yU_{m,n}(\cdot,y)\).  Taking difference quotients along a fixed
sequence \(h\to0\) also shows that this derivative, defined as zero on the
exceptional set, is strongly measurable.  Thus the slice map is locally
absolutely continuous in the Bochner sense with the asserted derivative.

For \(N>0\), the operator
\(A^{1/2}E([0,N])\) is bounded. Applying it to the strongly measurable slice
map gives a strongly measurable function of \(y\). For almost every \(y\),
these spectral truncations converge in \(L^2(M)\) to
\(A^{1/2}U_{m,n}(\cdot,y)\). Hence the latter map is strongly measurable.
The compact support in \eqref{eq:cutoff-support}, the local integrability of
\(y^{1-2s}\), and \eqref{eq:annular-gradient} make the extension energy
finite. All hypotheses of Lemma~\ref{lem:spectral-trace} are now verified.

The annuli are disjoint up to their boundary sets, which have zero
\(\widetilde\mu\)-measure.  Using \eqref{eq:geometric-extension-energy} and
summing the annular bounds,
we obtain
\begin{align*}
 \mathfrak E_s(U_{m,n})
 &\leq\sum_{k=m}^n\frac{\delta_k^2}{R_k^2}W(R_{k+1})
 =\sum_{k=m}^n\frac{\delta_k^2}{q_k}\\
 &=\frac1{S_{m,n}^2}\sum_{k=m}^nq_k
 =\frac1{S_{m,n}}.
\end{align*}
\end{proof}
\section{Proof of recurrence and consequences}
\label{sec:recurrence}

Assume \eqref{eq:radial-condition}. By
Lemma~\ref{lem:series-divergence}, every tail of \(\sum_kq_k\) diverges.
Choose integers \(m_j\to\infty\) and \(n_j>m_j\) such that
\[
 S_{m_j,n_j}\geq j.
\]
Let
\begin{equation}
 U_j=U_{m_j,n_j},
 \qquad
 f_j(x)=U_j(x,0)=\eta_{m_j,n_j}(d_g(o,x)).
 \label{eq:recurrence-cutoffs}
\end{equation}
The functions \(f_j\) satisfy
\[
 0\leq f_j\leq1,
 \qquad
 f_j\in C_c(M)\cap W^{1,2}(M),
 \qquad
 f_j=1\quad\text{on }B(o,R_{m_j}).
\]
For every fixed \(x\in M\), the last equality holds for all sufficiently
large \(j\), and hence \(f_j(x)\to1\).

Lemmas~\ref{lem:spectral-trace} and \ref{lem:cutoff-energy} give both
\(f_j\in\cD(A^{s/2})\) and
\begin{equation}
 \cE_s(f_j,f_j)=\|A^{s/2}f_j\|_2^2
 \leq C_s\mathfrak E_s(U_j)
 \leq\frac{C_s}{j}\longrightarrow0.
 \label{eq:energy-to-zero}
\end{equation}
The sequence \((f_j)\) therefore satisfies every condition in
\eqref{eq:recurrence-criterion}. The fractional form is recurrent, which
proves Theorem~\ref{thm:main}.

The theorem also yields the following volume-growth tests.

\begin{corollary}
\label{cor:polynomial-logarithmic}
Under the hypotheses of Theorem~\ref{thm:main}, the \(\alpha\)-process is
recurrent if either of the following bounds holds for all sufficiently large
\(r\):
\begin{enumerate}
 \item \(V(o,r)\leq Cr^d\) for constants \(C>0\) and \(d\leq\alpha\);
 \item \(V(o,r)\leq Cr^\alpha(\log r)^p\) for constants \(C>0\) and
 \(p\leq1\).
\end{enumerate}
\end{corollary}

\begin{proof}
In the first case, the integrand in \eqref{eq:radial-condition} is bounded
below by a positive constant times \(r^{\alpha-d-1}\), whose tail integral
diverges when \(d\leq\alpha\). In the second case, it is bounded below by a
positive constant times
\(r^{-1}(\log r)^{-p}\), whose tail integral diverges when \(p\leq1\).
Theorem~\ref{thm:main} applies.
\end{proof}

The polynomial assertion in Corollary~\ref{cor:polynomial-logarithmic}
recovers \cite[Theorem~16.2]{Grigoryan1999}.  The logarithmic assertion with
\(0<p\leq1\) permits larger critical growth than the bound \(Cr^\alpha\).

For positive functions of a scale parameter, write \(X\asymp Y\) if their
ratio is bounded above and below by positive scale-independent constants.
On \(M=\R^d\), one has \(V(o,R)\asymp R^d\), and the weighted product volume
has order
\[
 W(R)\asymp R^{d+2-\alpha}.
\]
The recurrence threshold becomes \(d+2-\alpha\leq2\), or \(d\leq\alpha\).
Euclidean space satisfies a relative Faber--Krahn inequality, so
\cite[Theorem~16.3]{Grigoryan1999} shows that the isotropic symmetric
\(\alpha\)-stable process is recurrent exactly when \(d\leq\alpha\).
Consequently, the polynomial-volume exponent in the criterion is sharp.

\section*{Acknowledgements}
The author thanks Professor Bobo Hua for helpful discussions.

\providecommand{\bysame}{\leavevmode\hbox to3em{\hrulefill}\thinspace}
\providecommand{\MR}{\relax\ifhmode\unskip\space\fi MR }
\providecommand{\MRhref}[2]{%
  \href{http://www.ams.org/mathscinet-getitem?mr=#1}{#2}
}
\providecommand{\href}[2]{#2}

\end{document}